\documentclass[reqno, a4paper]{amsart}

\usepackage{fixltx2e}
\usepackage[T1]{fontenc}
\usepackage[british]{babel}
\usepackage{amsfonts}
\usepackage{amssymb}
\usepackage{amsthm}
\usepackage{amsmath}
\usepackage{mathabx}
\usepackage{calrsfs}
\usepackage{hyperref}
\usepackage[all]{xy}
\usepackage{graphicx}

\theoremstyle{plain}
\newtheorem{theorem}[subsection]{Theorem}
\newtheorem{proposition}[subsection]{Proposition}

\theoremstyle{definition}
\newtheorem{example}[subsection]{Example}
\newtheorem{definition}[subsection]{Definition}
\newtheorem{remark}[subsection]{Remark}

\newenvironment{tfae}
{
\begin{enumerate}}
{\end{enumerate}}

\newcommand{\defn}{\textbf}
\newcommand{\noproof}{\hfill \qed}
\renewcommand{\implies}{$\Rightarrow$}

\newcommand{\links}{\langle}
\newcommand{\rechts}{\rangle}

\newcommand{\product}{\raisebox{0.2mm}{\ensuremath{\scriptstyle{\circ}}}}
\newcommand{\comp}{}

\newcommand{\N}{\ensuremath{\mathrm{N}}}
\newcommand{\Cl}{\ensuremath{\mathrm{Cl}}}
\newcommand{\I}{\ensuremath{\mathrm{I}}}
\newcommand{\Ker}{\ensuremath{\mathrm{Ker}}}

\newcommand{\V}{\ensuremath{\mathcal{V}}}
\newcommand{\C}{\ensuremath{\mathcal{C}}}

\newcommand{\Pt}{\ensuremath{\mathsf{Pt}}}

\newcommand{\SH}{{\rm (SH)}}
\newcommand{\SSH}{{\rm (SSH)}}

\def\pullback{
 \ar@{-}[]+R+<6pt,-1pt>;[]+RD+<6pt,-6pt>%
 \ar@{-}[]+D+<1pt,-6pt>;[]+RD+<6pt,-6pt>}

\def\pullbackdots{%
 \ar@{.}[]+R+<6pt,-1pt>;[]+RD+<6pt,-6pt>%
 \ar@{.}[]+D+<1pt,-6pt>;[]+RD+<6pt,-6pt>}

\newdir{>}{{}*:(1,-.2)@^{>}*:(1,+.2)@_{>}}
\newdir{<}{{}*:(1,+.2)@^{<}*:(1,-.2)@_{<}}
\newdir{ |>}{{}*!/-3.5pt/@{|}*!/-8pt/:(1,-.2)@^{>}*!/-8pt/:(1,+.2)@_{>}}
\newdir{ >>}{{}*!/8pt/@{|}*!/3.5pt/:(1,-.2)@^{>}*!/3.5pt/:(1,+.2)@_{>}}
\newdir{>>}{{}*!/3.5pt/:(1,-.2)@^{>}*!/3.5pt/:(1,+.2)@_{>}*!/7pt/:(1,-.2)@^{>}*!/7pt/:(1,+.2)@_{>}}
\newdir{ >}{{}*!/-8pt/@{>}}

\begin{document}

\thanks{This work was partially supported by the Centre for Mathematics of the University of Coimbra -- UID/MAT/00324/2013, by ESTG and
CDRSP from the Polytechnical Institute of Leiria --
UID/Multi/04044/2013, and by the grant number SFRH/BPD/69661/2010,
funded by the Portuguese Government through FCT/MEC and co-funded
by the European Regional Development Fund through the Partnership
Agreement PT2020}
\thanks{The second author is a Postdoctoral Researcher of the Fonds de la Recherche Scientifique--FNRS}
\thanks{The fourth author is a Research Associate of the Fonds de la
Recherche Scientifique--FNRS}

\author{N.~Martins-Ferreira}
\address[N.~Martins-Ferreira]{Departamento de
Matem\'atica, Escola Superior de Tecnologia e Gest\~ao, Centro
para o Desenvolvimento R\'apido e Sustentado do Produto, Instituto
Polit\'ecnico de Leiria, 2411--901 Leiria, Portugal}
\thanks{}
\email{martins.ferreira@ipleiria.pt}

\author{A.~Montoli}
\address[A.~Montoli]{CMUC, Universidade de Coimbra,
3001--501 Coimbra, Portugal\newline and\newline Institut de
Recherche en Math\'ematique et Physique, Universit\'e catholique
de Lou\-vain, chemin du cyclotron~2 bte~L7.01.02, 1348
Louvain-la-Neuve, Belgium}
\thanks{}
\email{montoli@mat.uc.pt}

\author{A.~Ursini}
\address[A.~Ursini]{DIISM, Department of Information Engineering and Mathematical Sciences, Universit\`a degli studi di Siena, 53100 Siena, Italy}
\email{aldo.ursini@unisi.it}

\author{T.~Van~der Linden}
\address[T.~Van~der Linden]{Institut de Recherche
en Math\'ematique et Physique, Universit\'e catholique de Louvain,
chemin du cyclotron~2 bte~L7.01.02, 1348 Louvain-la-Neuve,
Belgium}
\thanks{}
\email{tim.vanderlinden@uclouvain.be}

\keywords{Ideal; clot; zero-class; surjective left split relation;
pointed exact Mal'tsev category; \emph{Smith is Huq} condition.}


\title[What is an ideal a zero-class of?]{What is an ideal a zero-class of?}

\begin{abstract}
We characterise, in pointed regular categories, the ideals as the
zero-classes of surjective relations. Moreover, we study a
variation of the \emph{Smith is Huq} condition: two surjective
left split relations commute as soon as their zero-classes
commute.
\end{abstract}

\maketitle

\section{Introduction}

The description of congruences, and of some other relations, in
terms of their zero-classes is a very classical topic in universal
algebra. It led to the study of different notions of subalgebras
in pointed varieties; let us mention here the ones of
ideal~\cite{Higgins, Magari, Ursini1} and
clot~\cite{Agliano-Ursini}.

Later these notions have been considered in a categorical context
\cite{Janelidze-Marki-Ursini, Janelidze-Marki-Ursini2, MM-NC}.
Clots were characterised as zero-classes of internal reflexive
relations, and ideals were characterised as regular images of
clots. However, a characterisation of ideals as zero-classes of
suitable relations was still missing, both in universal and
categorical algebra.

The aim of the present paper is to fill this gap. We prove that,
in every pointed regular category, the ideals are the zero-classes
of what we call \emph{surjective relations}. Such is any relation
from an object $X$ to an object $Y$ where the projection on $Y$ is
a regular epimorphism. In fact, we can always choose a \emph{left
split} surjective relation to represent a given ideal, which means
that moreover the projection on $X$ is a split epimorphism. We
also show that, in general, it is not possible to describe ideals
by means of endorelations on an object $X$. The table at the end
of the introduction gives an overview of the description of all
the notions mentioned above in terms of zero-classes.

A related issue is to consider a variation of the so-called
\emph{Smith is Huq} condition, which says that two equivalence
relations on the same object commute in the Smith--Pedicchio
sense~\cite{Smith, Pedicchio} if and only if their zero-classes
commute in the Huq sense~\cite{Huq}. Our condition is then the
following: two semi-split surjective relations commute if and only
if their zero-classes (their associated ideals) commute. This
provides a conceptual interpretation of the \emph{admissibility}
condition introduced in~\cite{MF-PhD} and further explored
in~\cite{HVdL, MFVdL4}. We consider some equivalent and some
stronger conditions, and we compare them with the standard
\emph{Smith is Huq} condition.

The paper is organised as follows. In Section 2 we recall the
notions of ideal and clot, both from the universal and the
categorical-algebraic points of view, and we prove some stability
properties of ideals. In Section 3 we prove that ideals are
exactly zero-classes of surjective relations (or, equivalently, of
semi-split surjective relations) and we consider some concrete
examples. In Section~4 we study the above-mentioned variations of
the \emph{Smith is Huq} condition.

\begin{table}[h!]
\hrule\smallskip \resizebox{\textwidth}{!} {\begin{tabular}{ccccc}
\txt{any\\
(left split)\\
relation} & \txt{surjective\\ (left split)\\ relation} & \txt{reflexive relation} & \txt{equivalence relation} & \txt{effective\\ equivalence relation}\\\\
monomorphism & ideal & clot & normal monomorphism & kernel
\end{tabular}}
\smallskip\hrule\medskip
\caption{Several types of monomorphisms in pointed regular
categories}\label{overview}
\end{table}

\section{Ideals and clots}

The notion of \emph{ideal} was introduced in~\cite{Higgins} in the
context of groups with multiple operators (also called
$\Omega$-groups), and then extended in~\cite{Magari}---and further
studied in~\cite{Ursini1} and in subsequent papers---to varieties
of algebras with a constant $0$. We recall here the definition in
the case of \defn{pointed} varieties: those with a unique constant
$0$.

\begin{definition} \label{ideal universal algebra}
A term $t(x_1, \ldots, x_m, y_1, \ldots , y_n)$ in a pointed
variety $\V$ is said to be an \defn{ideal term} in $y_1$, \dots,
$y_n$ if $t(x_1, \ldots , x_m, 0, \ldots , 0) = 0$ is an identity
in $\V$. A subalgebra $I$ of an algebra $A$ in $\V$ is an
\defn{ideal} of $A$ if $t(x_1, \ldots, x_m, i_1, \ldots , i_n)$
belongs to $I$ for all $x_1$, \dots, $x_m \in A$, all $i_1$,
 \dots, $i_n \in I$ and every ideal term $t$.
\end{definition}

Later, as an alternative, in the paper~\cite{Agliano-Ursini} the
concept of \emph{clot} was introduced:

\begin{definition} \label{clot universal algebra}
A subalgebra $K$ of $A$ in $\V$ is called a \defn{clot} in $A$ if
\[ t(a_1, \ldots , a_m, 0, \ldots , 0) = 0
\]
and $k_1, \ldots , k_n \in K$ imply $t(a_1, \ldots , a_m, k_1,
\ldots , k_n) \in K$ for all $a_1$, \dots, $a_m$, $k_1$,~\dots,
$k_n$ in $A$ and every $(m+n)$-ary term function $t$ of $A$.
\end{definition}

It was shown in~\cite{Agliano-Ursini} that clots are exactly
$0$-classes of semi-congruences, that is, of those reflexive
relations which are compatible with all the operations in the
variety. Thus, for any algebra $A$ in any variety $\V$ there is an
inclusion
\[ \N(A) \subseteq \Cl(A) \subseteq \I(A), \]
where $\N(A)$ is the set of normal subalgebras of $A$ (that are
the $0$-classes of the congruences on $A$), $\Cl(A)$ is the set of
clots of $A$ and $\I(A)$ is the set of ideals.

All these notions were then studied in a categorical context (see
\cite{Janelidze-Marki-Ursini, Janelidze-Marki-Ursini2, MM-NC}).
Before recalling the categorical counterparts of the definitions
above, we need to introduce some terminology. The context that we
consider is the one of pointed regular categories.

\begin{definition}
Given a span
\begin{equation}\label{normalisation}
\vcenter{\xymatrix@!0@=4em{& R \ar[ld]_-{d} \ar[rd]^-{c} \\
X && Y}}
\end{equation}
a \defn{zero-class} of it is the arrow $i \colon {I \to Y}$ in the
pullback
\begin{equation}\label{zero-class}
\vcenter{\xymatrix@!0@=4em{
I \ar[r]^-{l} \pullback \ar[d]_-{i} & R \ar[d]^{\links d,c\rechts} \\
Y \ar[r]_-{\links0,1_{Y}\rechts} & X\times Y.}}
\end{equation}
\end{definition}

\begin{definition}
A \defn{normalisation} of~\eqref{normalisation} is the composite
$ck\colon{K\to X}$, where $k\colon {K\to R}$ is a kernel of $d$.
\end{definition}

Observe that, for our purposes, $(d,c)$ and $(c,d)$ are different
spans. If the span $(d,c)$ is a relation, which means that $d$ and
$c$ are jointly monomorphic, then its zero-class is a
monomorphism, since pullbacks preserve monomorphisms. Similarly,
the normalisation of a relation is a monomorphism, too. Of course
the zero-class and the normalisation of a span are unique up to
isomorphism, so (with abuse of terminology) we may talk about
``the'' zero-class and ``the'' normalisation. In fact, the two
procedures give the same result:

\begin{proposition}\label{zero-class = normalisation}
For any span $(d,c)$ its zero-class coincides with its
normalisation.
\end{proposition}
\begin{proof}
It is easily seen that the morphism $l$ in the
diagram~\eqref{zero-class} is a kernel of~$d$. As a consequence,
$i=cl$. On the other hand, any square such as~\eqref{zero-class}
in which $l=\ker(d)$ and $i=cl$ is a pullback.
\end{proof}

The second part of the proof also follows from the observation
that the zero-class and the normalisation of a span are unique up
to isomorphism.

\begin{definition} \label{normal subobject}
A \defn{normal subobject} of an object $A$ is the zero-class of an
equivalence relation on $A$.
\end{definition}

We observe that this notion is a generalisation of the notion of
\textbf{kernel} of a morphism: indeed, kernels are exactly
zero-classes of effective equivalence relations. It is also easy
to see that, in the pointed case, the definition above is
equivalent to the one introduced by Bourn in~\cite{Bourn2000}:
see~\cite{MM-NC} and Example~3.2.4, Proposition~3.2.12
in~\cite{Borceux-Bourn}.

\begin{definition} \label{clot categories}
A \defn{clot} of $A$ is the zero-class of a reflexive relation on
$A$.
\end{definition}

The original categorical definition of clot, given in
\cite{Janelidze-Marki-Ursini}, was different: roughly speaking, a
clot of an object $A$ was defined as a subobject which is
invariant under the \emph{conjugation action} on $A$. However, the
two definitions are equivalent, as already observed in
\cite{Janelidze-Marki-Ursini}.

The following categorical definition of ideal was proposed in
\cite{Janelidze-Marki-Ursini2}. It was observed in
\cite{Janelidze-Marki-Ursini} that, in the varietal case, it
coincides with Definition~\ref{ideal universal algebra} above.

\begin{definition}
A monomorphism $i\colon {I\to Y}$ is an \defn{ideal} if there
exists a commutative square
\begin{equation}\label{ideal}
\vcenter{\xymatrix@!0@=4em{K \ar@{{ >}->}[d]_-{k} \ar@{-{>>}}[r]^{q}& I \ar@{{ >}->}[d]^{i} \\
X \ar@{-{>>}}[r]_-{p} & Y}}
\end{equation}
in which $p$ and $q$ are regular epimorphisms and $k$ is a clot.
In other words, an ideal is the \defn{regular image} of a clot.
\end{definition}

The following fact was already observed in
\cite[Corollary~3.1]{Janelidze-Marki-Ursini2}:

\begin{proposition}
Every ideal is the regular image of a kernel along a regular
epimorphism.
\end{proposition}
\begin{proof}
Proposition~\ref{zero-class = normalisation} tells us that the
morphism $k$ in Diagram~\eqref{ideal} is of the form $cl$ for some
kernel $l$ and some split epimorphism $c$. The claim now follows,
since a composite of two regular epimorphisms in a regular
category is still a regular epimorphism.
\end{proof}

The first aim of this paper is to characterise the ideals as the
zero-classes of suitable relations. Before doing that, we prove
some stability properties of clots and ideals.

\begin{proposition}\label{stability clots}
Clots are stable under pullbacks, and an intersection of two clots
is still a clot.
\end{proposition}
\begin{proof}
Suppose that $k\colon{K\to X}$ is the zero-class of a reflexive
relation $(R,d,c)$ on $X$ and consider a morphism $f\colon{Y\to
X}$. Pull back $k$ along $f$, and $\links d,c\rechts \colon {R\to
X\times X}$ along $f\times f\colon{Y\times Y\to X\times X}$, to
obtain the commutative cube
\[
\xymatrix@!0@=4em{K' \ar@{{ >}->}[dd]_-{k'} \ar@{-{>}}[rr] \ar[rd] && K \ar@{{ >}->}[dd]_(.25){k}|{\hole} \ar@{->}[rd] \\
& R' \ar@{-{>}}[rr] \ar@{{ >}->}[dd]_(.25){\links d',c'\rechts} && R \ar@{{ >}->}[dd]^-{\links d,c\rechts} \\
Y \ar@{-{>}}[rr]^(.75){f}|{\hole} \ar[rd]_-{\links 0,1_{Y}\rechts} && X \ar[rd]_(.3){\links 0,1_{X}\rechts}\\
& Y\times Y \ar[rr]_-{f\times f} && X\times X.}
\]
Note that $(R',d',c')$ is a reflexive relation on $Y$. Since the
front, back, and right hand side faces of this cube are all
pullbacks, also its left hand side face is a pullback. This means
that $k'$ is the zero-class of~$R'$, so that the pullback $k'$ of
$k$ is a clot.

Now consider two clots $k\colon{K\to X}$ and $l\colon{L\to X}$ on
an object $X$, the respective zero-classes of the reflexive
relations $(R,d,c)$ and $(S,d',c')$ on $X$. Consider the cube
\[
\xymatrix@!0@=4em{K\cap L \ar@{{ >}->}[dd] \ar@{-{>}}[rr] \ar@{{ >}->}[rd] && R\cap S \ar@{{ >}->}[dd]|{\hole} \ar@{{ >}->}[rd] \\
& K \ar@{-{>}}[rr] \ar@{{ >}->}[dd]_(.25){k} && R \ar@{{ >}->}[dd]^-{\links d,c\rechts} \\
L \ar@{-{>}}[rr]|{\hole} \ar@{{ >}->}[rd]_-{l} && S \ar@{{ >}->}[rd]_(.3){\links d',c'\rechts}\\
& X \ar[rr]_-{\links 0,1_{X}\rechts} && X\times X,}
\]
in which all faces are pullbacks. We see that the monomorphism
${K\cap L\to X}$ is the zero-class of the reflexive relation
${R\cap S\to X\times X}$ on $X$. In other words, the intersection
of the clots $k$ and $l$ is still a clot.
\end{proof}

\begin{proposition}\label{stability}
Ideals are stable under:
\begin{enumerate}
\item regular images; \item pullbacks; \item compositions with
product inclusions.
\end{enumerate}
Furthermore,
\begin{enumerate}
\setcounter{enumi}{3} \item an intersection of an ideal and a clot
is an ideal; \item an intersection of two ideals is an ideal.
\end{enumerate}
\end{proposition}
\begin{proof}
(a) is immediate from the definition. (b) holds because, given a
morphism $f\colon{Y'\to Y}$ and an ideal $i$ which is a regular
image of a clot $k$ as in~\eqref{ideal}, we may consider the
commutative cube
\[
\xymatrix@!0@=4em{K' \ar@{->}[dd] \ar@{.{>>}}[rr]^-{q'} \ar@{{ >}->}[rd]^{k'} && I' \ar[dd]|{\hole} \ar@{{ >}->}[rd]^-{i'} \\
& X' \pullback \ar@{.{>>}}[rr]^(.25){p'} \ar@{->}[dd]_(.25){f'} && Y' \ar@{->}[dd]^-{f} \\
K \ar@{-{>>}}[rr]^(.25){q}|{\hole} \ar@{{ >}->}[rd]_-{k} && I \ar@{{ >}->}[rd]_-{i}\\
& X \ar@{-{>>}}[rr]_-{p} && Y}
\]
in which the front, left and right squares are pullbacks by
construction. It follows that the back square is also a pullback,
and the dotted arrows $p'$ and $q'$ are regular epimorphisms.
Furthermore, the monomorphism $k'$ is a clot by
Proposition~\ref{stability clots}. As a consequence, the pullback
$i'$ of $i$ along $f$ is an ideal, as a regular image of the
clot~$k'$.

For the proof of (c), recall that kernels compose with product
inclusions: if $k\colon K\to X$ is the kernel of $f\colon X\to
X'$, then $\links 1_{X},0\rechts k=\links k,0\rechts\colon {K\to
X\times W}$ is the kernel of $f\times 1_W\colon{X\times W\to
X'\times W}$ for any object $W$. If now $i$ is an ideal as
in~\eqref{ideal}, then $\links 1_{Y},0\rechts i=\links
i,0\rechts\colon {I\to Y\times W}$ is the direct image of $\links
k,0\rechts$ along the regular epimorphism $f\times 1_{W}$.

For the proof of (d), suppose $i$ is an ideal as in~\eqref{ideal}
and $l\colon {L\to Y}$ is a clot. We consider the commutative cube
\[
\xymatrix@!0@=4em{K\cap L' \ar@{{ >}->}[dd] \ar@{.{>}}[rr] \ar@{{ >}->}[rd] && I\cap L \ar@{{ >}->}[dd]|{\hole} \ar@{{ >}->}[rd] \\
& L' \pullback \ar@{-{>>}}[rr] \ar@{{ >}->}[dd]_(.25){l'} && L \ar@{{ >}->}[dd]^-{l} \\
K \ar@{-{>>}}[rr]^(.25){q}|{\hole} \ar@{{ >}->}[rd]_-{k} && I \ar@{{ >}->}[rd]_(.3){i}\\
& X \ar@{-{>>}}[rr]_-{p} && Y}
\]
in which the front, left and right squares are pullbacks by
construction. Then the back square is also a pullback, so that the
dotted arrow is a regular epimorphism. Since the monomorphism $
l'$, and thus also ${K\cap L'\to X}$, are still clots by
Proposition~\ref{stability clots}, this proves that the
intersection ${I\cap L\to Y}$ is an ideal.

For the proof of (e), suppose that both $i$ and $l$ are ideals.
Repeating the above construction, through (b) and (d) we see that
${K\cap L'\to X}$ is an ideal, as the intersection of the clot $k$
with the ideal $l'$. The result now follows from (a).
\end{proof}

\section{Ideals and semi-split surjective relations}

In order to characterise ideals as zero-classes, we shall be
interested in spans where one of the legs is a regular or even a
split epimorphism.

\begin{definition}
A \defn{left split span} from $X$ to $Y$ is a diagram
\begin{equation}\label{left split span}
\vcenter{\xymatrix@!0@=4em{& R \ar@<-.5ex>[ld]_-{d} \ar[rd]^-{c} \\
X \ar@<-.5ex>[ru]_-{e} && Y}}
\end{equation}
where $de=1_{X}$. A left split span $(d,c,e)$ is called a
\defn{left split relation} when the span $(d,c)$ is jointly monomorphic.
\end{definition}

\begin{proposition}\label{monomorphisms}
For a morphism $i\colon I\to Y$, the following conditions are
equivalent:
\begin{tfae}
\item $i$ is a monomorphism;

\item $i$ is the zero-class of a left split relation;

\item $i$ is the zero-class of a relation on~$Y$.
\end{tfae}
\end{proposition}

\begin{proof}
For the equivalence between (i) and (ii) it suffices to take
$X=0$, and for the one between (i) and (iii) we consider the span
$(0\colon{I\to Y},i\colon {I\to Y})$. In both cases the span at
hand is a relation if and only if $i$ is a monomorphism.
\end{proof}

\begin{definition}
A \defn{surjective span} from $X$ to $Y$ is a diagram
\[
\xymatrix@!0@=4em{& R \ar[ld]_-{d} \ar@{-{>>}}[rd]^-{c} \\
X && Y}
\]
where $c$ is a regular epimorphism. A surjective span $(d,c)$ is
called a \defn{surjective relation} when the span $(d,c)$ is
jointly monomorphic.
\end{definition}

Sometimes we consider both conditions together and talk about
\defn{surjective left split} spans or relations.

We are now ready to prove our main result.

\begin{theorem}\label{theorem}
In any pointed regular category, for any morphism $i\colon {I\to
Y}$, the following are equivalent:
\begin{tfae}
\item $i$ is an ideal; \item $i$ is the zero-class of a surjective
left split relation; \item $i$ is the zero-class of a surjective
relation.
\end{tfae}
\end{theorem}
\begin{proof}
To prove (i) \implies\ (ii), suppose that $i$ is an ideal as
in~\eqref{ideal} above, where $k$ is the zero-class of a reflexive
relation $(R,d,c,e)$. We consider the commutative cube
\[
\xymatrix@!0@=4em{K \ar@{{ >}->}[dd]_-{k} \ar@{-{>>}}[rr]^{q} \ar[rd] && I \ar@{{ >}->}[dd]_(.25){i}|{\hole} \ar@{.>}[rd] \\
& R \ar@{.{>>}}[rr]^(.25){q'} \ar@{{ >}->}[dd]_(.25){\links d,c\rechts} && S \ar@{{ >}.>}[dd]^-{\links d',c'\rechts} \\
X \ar@{-{>>}}[rr]^(.25){p}|{\hole} \ar[rd]_-{\links 0,1_{X}\rechts} && Y \ar[rd]_(.3){\links 0,1_{Y}\rechts}\\
& X\times X \ar@{-{>>}}[rr]_-{1_{X}\times p} && X\times Y}
\]
in which $S$ is the regular image of $R$ along $1_{X}\times p$ and
${I\to S}$ is induced by functoriality of image factorisations. We
have to show that the square on the right is a pullback. Let the
square on the left
\[
\vcenter{\xymatrix@!0@=4em{P \ar[r] \ar[d] \pullback & S \ar@{{ >}->}[d]\\
Y \ar[r]_-{\links 0,1_{Y}\rechts} & X\times Y}} \qquad\qquad\qquad
\vcenter{\xymatrix@!0@=4em{K \ar@{{ >}->}[dd]_-{k} \ar@{.>}[rr] \ar[rd] && P \ar@{{ >}->}[dd]|{\hole} \ar@{->}[rd] \\
& R \ar@{-{>>}}[rr]^(.25){q'} \ar@{{ >}->}[dd]_(.25){\links d,c\rechts} && S \ar@{{ >}->}[dd]^-{\links d',c'\rechts} \\
X \ar@{-{>>}}[rr]^(.25){p}|{\hole} \ar[rd]_-{\links 0,1_{X}\rechts} && Y \ar[rd]_(.3){\links 0,1_{Y}\rechts}\\
& X\times X \ar@{-{>>}}[rr]_-{1_{X}\times p} && X\times Y}}
\]
be the pullback in question. The induced arrow $f \colon I\to P$
is an isomorphism. Indeed it is a monomorphism since $i$ is.
Moreover, the bottom and left squares in the cube are pullbacks,
and so the dotted arrow $K\to P$ is a regular epimorphism, being a
pullback of the regular epimorphism $q'$. Then $f$ is a regular
epimorphism, hence an isomorphism. Note that $d'$ is split
by~$q'e$ and $c'$ is a regular epimorphism because $pc=c'q'$ is.

(ii) \implies\ (iii) is obvious. For the proof of (iii) \implies\
(i), let $i\colon{I\to Y}$ be the zero-class \eqref{zero-class} of
a surjective relation $(d,c)$. Consider the pullback
\[
\vcenter{\xymatrix@!0@=5em{
T \ar@{.>}[r] \pullbackdots \ar@{{ >}.>}[d]_-{\links d',c'\rechts} & R \ar@{{ >}->}[d]^{\links d,c\rechts} \\
R\times R \ar[r]_-{d\times c} & X\times Y}}
\]
of $\links d,c\rechts$ and $d\times c$, which defines a reflexive
relation $(T,d',c',e')$ on $R$, where $e'$ is $\links\links
1_{R},1_{R}\rechts,1_{R}\rechts$. We prove that $i$ is the regular
image of the zero-class $k$ of $T$ along the regular epimorphism
$c$ as in the square on the left.
\[
\vcenter{\xymatrix@!0@=4em{K \ar@{{ >}->}[d]_-{k} \ar@{-{>>}}[r]^{q}& I \ar@{{ >}->}[d]^{i} \\
R \ar@{-{>>}}[r]_-{c} & Y}} \qquad\qquad\qquad
\vcenter{\xymatrix@!0@=4em{K \ar@{{ >}->}[dd]_-{k} \ar@{.{>}}[rr]^{q} \ar[rd] && I \ar@{{ >}->}[dd]_(.25){i}|{\hole} \ar@{->}[rd] \\
& T \ar@{-{>>}}[rr] \ar@{{ >}->}[dd]_(.25){\links d',c'\rechts} && R \ar@{{ >}->}[dd]^{\links d,c\rechts} \\
R \ar@{-{>>}}[rr]^(.25){c}|{\hole} \ar[rd]_-{\links 0,1_{R}\rechts} && Y \ar[rd]_(.3){\links 0,1_{Y}\rechts}\\
& R\times R \ar@{-{>>}}[rr]_-{d\times c} && X\times Y}}
\]
Here it suffices to consider the cube on the right, noting that
$q$ is a regular epimorphism because all vertical squares are
pullbacks and $c$ is a regular epimorphism by assumption.
\end{proof}

\begin{remark}
Consider a pointed variety of universal algebras $\V$ and let $A
\in \V$. According to the previous theorem, a subalgebra $I$ of
$A$ is an ideal of $A$ if and only if there is a surjective
relation~$R$ for which $I$ is the zero-class of $R$. In other
words, $I$ is an ideal if and only if there exists a subalgebra
$R$ of $B\times A$, for some $B\in \V$, such that the second
projection is surjective and $a\in I$ if and only if $(0,a)\in R$.
A direct proof of this is in fact pretty simple. That such a
zero-class is an ideal is trivial from Higgins' definition of
ideals by means of ideal terms (Definition \ref{ideal universal
algebra}). For the converse, assume that $I$, as an ideal of $A$,
is the image $f(K)$ of a clot $K$ of some $B\in \V$ under a
surjective homomorphism $f\colon{B\to A}$. Then $K$ is the
zero-class of some reflexive subalgebra $S$ of $B\times B$. One
easily sees that the relational product $f\product S$, where now
$f$ means ``the set-theoretic graph of the mapping $f$'', is a
surjective relation, whose zero-class is exactly $I$.
\end{remark}

As the following example shows, in general it is not possible to
see every ideal as a zero-class of a surjective endorelation. We
are grateful to Sandra Mantovani for suggestions concerning this
example.

\begin{example}
Let $\V$ be the variety defined by a unique constant $0$ and a
binary operation $s$ satisfying just the identity $s(0,0)=0$. In
this variety, ideal terms are all ``pure'': in any term
$t(x_1,\dots,x_m,y_1,\dots,y_n)$ which is an ideal term in $y_1$,
\dots, $y_n$, necessarily $m=0$. Therefore all subalgebras are
ideals. Consider then the three element algebra $A=\{0,1,a\}$,
with $s(a,1)=s(1,a)=s(a,a)=a$, and $s(x,y)=0$ otherwise.
$C=\{0,1\}$ is a subalgebra, and we have that $s(a,0)=0$ lies in
$C$, but $s(a,1)=a$ does not belong to $C$. Hence $C$ is an ideal,
but not a clot. Suppose that there exists a surjective relation
$R$ on $A$ such that $C$ is its zero-class. Then there should
exist $x \in A$ such that $xRa$. But $x$ cannot be $0$, because $a
\notin C$. $1Ra$ is impossible, too, because otherwise $s(1,1) R
s(a,a)$, while $s(1,1)=0$ and $s(a,a)=a$. Similarly, $aRa$ is
impossible, otherwise $s(0,a) R s(1,a)$, while $s(0,a)=0$ and
$s(1,a)=a$. Hence such a surjective endorelation $R$ does not
exist.
\end{example}

We conclude this section with the following observation. It is
well known~\cite{MM-NC} that, in any pointed exact Mal'tsev
category, ideals and kernels coincide. Theorem~\ref{theorem}
provides us with the following quick argument.

\begin{proposition}\label{pointed exact Mal'tsev}
In any pointed exact Mal'tsev category, ideals and kernels
coincide.
\end{proposition}
\begin{proof}
Let $i\colon {I\to Y}$ be the zero-class of a surjective left
split relation $(d,c,e)$ as in~\eqref{left split span}.
Theorem~5.7 in~\cite{Carboni-Kelly-Pedicchio} tells us that the
pushout of $d$ and $c$ is also a pullback; as a consequence, $i$
is the kernel of the pushout $c_{*}(d)$ of $d$ along $c$.
\end{proof}

\section{The \emph{Smith is Huq} condition}

From now on we work in a category which is pointed, regular and
weakly Mal'tsev~\cite{NMF1}. We first recall

\begin{definition} \label{weakly Mal'tsev category}
A finitely complete category is weakly Mal'tsev if, for any
pullback of the form:
\begin{equation} \label{pullback weakly Mal'tsev}
\vcenter{\xymatrix@!0@=4em{ A \times_B C \ar@<-.5ex>[d]_{\pi_1}
\ar@<-.5ex>[r]_-{\pi_2} & C \ar@<-.5ex>[d]_g \ar@<-.5ex>[l]_-{e_2} \\
A \ar@<-.5ex>[u]_{e_1} \ar@<-.5ex>[r]_f & B, \ar@<-.5ex>[u]_s
\ar@<-.5ex>[l]_r }}
\end{equation}
where $fr = 1_B = gs$, the morphisms $e_{1}=\links1_A,s\comp f
\rechts$ and $e_{2}=\links r\comp g,1_C \rechts$, induced by the
universal property of the pullback, are jointly epimorphic.
\end{definition}

We observe that any finitely complete Mal'tsev category
\cite{Carboni-Lambek-Pedicchio} is weakly Mal'tsev. Indeed, in
\cite{Bourn1996} it was proved that a finitely complete category
is Mal'tsev if, for any pullback of the form \eqref{pullback
weakly Mal'tsev}, the morphisms $e_{1}=\links1_A,s\comp f \rechts$
and $e_{2}=\links r\comp g,1_C \rechts$ are jointly
\emph{strongly} epimorphic. In particular, every Mal'tsev variety
\cite{Malcev} is a weakly Mal'tsev category. In \cite{Nelson1} it
is shown that the variety of distributive lattices is weakly
Mal'tsev. In \cite{Nelson2} several other examples are given,
amongst which the variety of commutative monoids with cancelation.

In the context of weakly Mal'tsev categories, we say that two left
split spans $(f,\alpha,r)$ and $(g,\gamma,s)$ from $B$ to~$D$ as
in
\begin{equation}\label{adm}
\vcenter{\xymatrix@!0@=4em{A \ar@<.5ex>[r]^-{f} \ar[rd]_-{\alpha}
& B \ar@<.5ex>[l]^-{r} \ar@<-.5ex>[r]_-{s}
 & C \ar@<-.5ex>[l]_-{g} \ar[ld]^-{\gamma}\\
& D}}
\end{equation}
\defn{centralise each other} or \defn{commute} when there exists a
(necessarily unique) morphism $\varphi\colon{A\times_B C\to D}$,
called \defn{connector} from the pullback
\begin{equation*}\label{kite}
\vcenter{\xymatrix@!0@=3em{ & C \ar@<.5ex>[ld]^-{e_2}
\ar@<-.5ex>[rd]_-{g}
\ar@/^/[rrrd]^-{\gamma} \\
A\times_{B}C \ar@<.5ex>[ru]^-{\pi_2} \ar@<-.5ex>[rd]_-{\pi_1} && B
\ar@<.5ex>[ld]^-{r} \ar@<-.5ex>[lu]_-{s}
 \ar@{.>}[rr]|-{\beta} && D\\
& A \ar@<.5ex>[ru]^-{f} \ar@<-.5ex>[lu]_-{e_1}
\ar@/_/[urrr]_-{\alpha}}}
\end{equation*}
of $f$ and $g$ to the object $D$ such that $\varphi\comp
e_1=\alpha$ and $\varphi \comp e_2=\gamma$. Note that, when this
happens, $\gamma s=\alpha r$; we denote this morphism by
$\beta\colon B\to D$. In other words, the existence of $\beta$ is
a necessary condition for the given left split spans to centralise
each other. The condition for two left split spans to commute was
called \emph{admissibility} in~\cite{MF-PhD}. There it was
implicit that such a condition deals with a certain type of
commutativity, but it was not possible to express precisely
\emph{what} commutes. Our new interpretation makes it clear that
the admissibility condition is just the commutation of left split
spans.

If we take $\beta = 1_B$, we immediately recover the notion of
commutativity of reflexive graphs in the Smith--Pedicchio sense:

\begin{proposition}
Two reflexive graphs commute in the Smith--Pedicchio sense if and
only if they commute in the above sense.\noproof
\end{proposition}

We recall that the commutativity of equivalence relations was
first introduced by Smith in~\cite{Smith} for Mal'tsev varieties,
and then extended by Pedicchio~\cite{Pedicchio} to Mal'tsev
categories. However, weakly Mal'tsev categories are a suitable
setting for the definition (because the connector, as defined
above, is unique), and the commutativity can be defined, as above,
just for reflexive graphs.

If, in the diagram \eqref{adm}, we take $B = 0$, we get the
definition of commutativity of two morphisms in the Huq
sense~\cite{Huq}: two morphisms $\alpha \colon {A \to D}$ and
$\gamma \colon {C \to D}$ \defn{commute} when there exists a
(necessarily unique) morphism $\varphi \colon {A \times C \to D}$,
called the
\defn{cooperator} of $\alpha$ and $\gamma$, such that
\[
\varphi \links 1_A, 0\rechts = \alpha \qquad \text{and}\qquad
\varphi \links 0, 1_C \rechts = \gamma.
\]

A pointed regular weakly Mal'tsev category satisfies the
\defn{Smith is Huq} condition \cite{MFVdL1}, shortly denoted by \SH, when a
pair of equivalence relations over the same object commutes as
soon as their zero-classes do. (The converse is always true). We
observe that the \SH\ condition has the following interesting
consequence. We recall that an object $A$ is
\defn{commutative} if its identity commutes with itself (in the
Huq sense); it is \defn{abelian} if it has an internal abelian
group structure.

\begin{proposition}
If \SH\ is satisfied, then every commutative object is abelian.
\end{proposition}

\begin{proof}
The identity $1_X$ of an object $X$ is the normalisation of the
indiscrete relation~$\nabla_X$. If $X$ is commutative, then $1_X$
commutes with itself; by the \SH\ condition, the relation
$\nabla_X$ commutes with itself, too. This situation is
represented by the following diagram:
\[
\xymatrix@!0@=3em{ & X \times X \times X \ar[dd]_p \ar@<-.5ex>[dl] \ar@<-.5ex>[dr] & \\
X \times X \ar@<-.5ex>[ur] \ar@<-.5ex>[dr] & & X \times X
\ar@<-.5ex>[ul] \ar@<-.5ex>[dl] \\
& X. \ar@<-.5ex>[ul] \ar@<-.5ex>[ur] & }
\]
The connector $p \colon X \times X \times X \to X$ is then an
internal Mal'tsev operation on~$X$. To conclude the proof it
suffices to observe that, in a pointed category, an object is
endowed with an internal Mal'tsev operation if and only if it is
endowed with an internal abelian group
structure~\cite[Proposition~2.3.8]{Borceux-Bourn}.
\end{proof}

Our aim is to study the condition obtained by replacing
equivalence relations and normal subobjects in \SH\ by surjective
left split relations and ideals. In order to do so, we start by
introducing some terminology. We call a morphism
\defn{ideal-proper} when its regular image is an ideal; we say that a
cospan is \defn{ideal-proper} when so are the morphisms of which
it consists.

In a pointed finitely complete category $\C$, given an object $B$,
the category $\Pt_{B}(\C)$ of so-called \defn{points over $B$} is
the category whose objects are pairs \linebreak $(p\colon {E\to
B},s\colon {B\to E})$ where $ps=1_{B}$. A morphism
\[
(p\colon {E\to B},s)\to (p'\colon{E'\to B},s')
\]
in $\Pt_{B}(\C)$ is a morphism $f\colon E\to E'$ in $\C$ such that
$p'f=p$ and $fs=s'$. We have, for any $B$, a functor (called the
\defn{kernel functor}) $\Ker_B \colon \Pt_B(\C) \to \C$
associating with every split epimorphism its kernel. We can now
formulate the main result of this section.

\begin{theorem}\label{Theorem (SH)}
In any pointed, regular and weakly Mal'tsev category $\C$, the
following are equivalent:
\begin{tfae}
\item for every object $B$ in $\C$, the kernel functor $\Ker_B
\colon{\Pt_{B}(\C)\to \C}$ reflects Huq-commutativity of
ideal-proper cospans; \item a pair of surjective left split
relations over the same objects commutes as soon as their
zero-classes do; \item a pair of surjective left split spans over
the same objects commutes as soon as their zero-classes do.
\end{tfae}
\end{theorem}

\begin{proof}
The equivalence between conditions (ii) and (iii) is proved just
by taking regular images. In order to prove that (i) and (ii) are
equivalent, given a pair of surjective left split relations over
the same object, we rewrite Diagram~\eqref{adm} in the shape
\begin{equation} \label{cospan in points}
\vcenter{\xymatrix@!0@=5em{A \ar@<-.5ex>[rd]_(.7){f}
\ar[r]^-{\links\alpha,f\rechts} & D\times B
\ar@<-.5ex>[d]_(.3){\pi_{B}}
 & C \ar@<.5ex>[dl]^(.7){g} \ar[l]_-{\links\gamma,g\rechts}\\
& B \ar@<-.5ex>[u]_(.7){\links\beta,1_{B}\rechts}
\ar@<-.5ex>[ul]_(.3){r} \ar@<.5ex>[ur]^(.3){s}}}
\end{equation}
and consider it as a cospan $(\links\alpha,f\rechts,
\links\gamma,g\rechts)$ in $\Pt_{B}(\C)$. Let us prove that this
cospan is ideal-proper. To do that, it suffices to notice that
$\links \alpha,f\rechts$ is the composite of the kernel $\links
1_{A},f\rechts\colon A\to A\times B$ with the regular epimorphism
$\alpha\times 1_{B}\colon {A\times B\to D\times B}$. Indeed, the
outer square in the diagram
\[
\xymatrix@!0@=4em{A \ar[rr]^-{\links 1_{A},f\rechts} \ar[dd]_-{f} \ar@<-.5ex>[rd]_(.6){f} && A\times B \ar@<-.5ex>[ld]_(.6){\pi_{B}} \ar[dd]^{f\times 1_{B}} \\
& B \ar@<-.5ex>[rd]_-{\links 1_{B},1_{B}\rechts} \ar@<-.5ex>[ul]_(.4){r} \ar@<-.5ex>[ur]_(.3){\links r,1_{B}\rechts} \ar@{=}@<-.5ex>[dl] \\
B \ar[rr]_-{\links 1_{B},1_{B}\rechts} \ar@<-.5ex>@{=}[ru] &&
B\times B \ar@<-.5ex>[lu]_(.6){\pi_{2}}}
\]
is a pullback in $\Pt_{B}(\C)$. The same is true for $\links
\gamma,g\rechts$. To conclude the proof of the equivalence between
(i) and (ii) it suffices then to observe that applying the kernel
functor $\Ker_B$ to the cospan \eqref{cospan in points} gives the
normalisations of the two surjective split relations.
\end{proof}

It is immediately seen that condition (i) above is equivalent to
the condition that, for every morphism $p\colon {E\to B}$ in $\C$,
the pullback functor
\[
p^{*}\colon{\Pt_{B}(\C)\to \Pt_{E}(\C)}
\]
---which sends every split epimorphism over $B$ into its pullback
along $p$---reflects Huq-commutativity of ideal-proper cospans. In
the same way as for the previous theorem, it can be shown that
also the following conditions are equivalent.

\begin{proposition} \label{Proposition (W)}
In any pointed, regular and weakly Mal'tsev category $\C$, the
following are equivalent:
\begin{tfae}
\setcounter{enumi}{3} \item for every object $B$ in $\C$, the
kernel functor $\Ker_B \colon{\Pt_{B}(\C)\to \C}$ reflects
Huq-commutativity of cospans; \item a pair of left split relations
over the same objects commutes as soon as their zero-classes do;
\item a pair of left split spans over the same objects commutes as
soon as their zero-classes do.\noproof
\end{tfae}
\end{proposition}

Again, condition (iv) can be expressed equivalently in terms of
all pullback functors $p^{*}\colon{\Pt_{B}(\C)\to \Pt_{E}(\C)}$.
Note that the conditions (iv)--(vi) are stronger then (i)--(iii):
this is easily seen by comparing conditions (i) and (iv). Indeed,
condition (iv) requires that the kernel functors reflect
commutativity of a wider class of cospans.

We conclude by observing that Theorem~\ref{Theorem (SH)} is a
generalisation of~\cite[Proposition~2.5]{MFVdL3} and
Proposition~\ref{Proposition (W)} is a generalisation
of~\cite[Proposition~3.1]{MFVdL3}, which have been proved for
pointed exact Mal'tsev categories (see also Theorem~2.1 in
\cite{BMFVdL}). Indeed, in such categories ideals coincide with
kernels (Proposition~\ref{pointed exact Mal'tsev}). In particular,
in a pointed exact Mal'tsev category the conditions (i)--(iii) are
equivalent to \SH, while (iv)--(vi) are equivalent to the stronger
condition \SSH---see~\cite{MFVdL3}.

\section{Acknowledgement}
We would like to thank the referee for valuable comments helping
us to improve the manuscript.


\providecommand{\noopsort}[1]{}
\providecommand{\bysame}{\leavevmode\hbox
to3em{\hrulefill}\thinspace}
\providecommand{\MR}{\relax\ifhmode\unskip\space\fi MR }
\providecommand{\MRhref}[2]{%
 \href{http://www.ams.org/mathscinet-getitem?mr=#1}{#2}
} \providecommand{\href}[2]{#2}

\end{document}